\documentclass[11pt,twoside]{amsart}
\usepackage{color}
\usepackage{hyperref}
\usepackage[latin1]{inputenc}
\usepackage[OT1]{fontenc}
\usepackage{amsmath}
\usepackage{amsthm}
\usepackage{amssymb}
\usepackage[all]{xy}
\usepackage{xy}
\usepackage{ifthen} 
\usepackage{hyperref} 
\usepackage{graphicx}
\usepackage{paralist}
\usepackage{stmaryrd}

\newtheorem{theorem}{Theorem}[section]

\newtheorem{proposition}[theorem]{Proposition}

\newtheorem{lemma}[theorem]{Lemma}

\numberwithin{equation}{section}

\theoremstyle{definition}
\newtheorem{definition}[theorem]{Definition}
\newtheorem{remark}[theorem]{Remark}

\newcommand{\CC}{\mathbb{C}}

\newcommand{\NN}{\mathbb{N}}
\newcommand{\PP}{\mathbb{P}}

\newcommand{\RR}{\mathbb{R}}
\newcommand{\ZZ}{\mathbb{Z}}

\newcommand{\TT}{\mathbb{T}}

\renewcommand{\to}{\xymatrix@1@=15pt{\ar[r]&}}
\renewcommand{\rightarrow}{\xymatrix@1@=15pt{\ar[r]&}}
\renewcommand{\mapsto}{\xymatrix@1@=15pt{\ar@{|->}[r]&}}
\renewcommand{\twoheadrightarrow}{\xymatrix@1@=15pt{\ar@{->>}[r]&}}
\renewcommand{\hookrightarrow}{\xymatrix@1@=15pt{\ar@{^(->}[r]&}}
\newcommand{\congpf}{\xymatrix@1@=15pt{\ar[r]^-\sim&}}

\usepackage{color}

\begin{document}

\newboolean{xlabels} 
\newcommand{\xlabel}[1]{ 
                        \label{#1} 
                        \ifthenelse{\boolean{xlabels}} 
                                   {\marginpar[\hfill{\tiny #1}]{{\tiny #1}}} 
                                   {} 
                       } 
\setboolean{xlabels}{false} 

\title[On the dynamical degrees of reflections on cubic fourfolds]{On the dynamical degrees of reflections on cubic fourfolds}

\author[B\"ohning]{Christian B\"ohning$^1$}
\address{Christian B\"ohning, Fachbereich Mathematik der Universit\"at Hamburg\\
Bundesstra\ss e 55\\
20146 Hamburg, Germany}
\email{christian.boehning@math.uni-hamburg.de}

\author[Bothmer]{Hans-Christian Graf von Bothmer}
\address{Hans-Christian Graf von Bothmer, Fachbereich Mathematik der Universit\"at Hamburg\\
Bundesstra\ss e 55\\
20146 Hamburg, Germany}
\email{hans.christian.v.bothmer@uni-hamburg.de}

\author[Sosna]{Pawel Sosna$^2$}
\address{Pawel Sosna, Fachbereich Mathematik der Universit\"at Hamburg\\
Bundesstra\ss e 55\\
20146 Hamburg, Germany}
\email{pawel.sosna@math.uni-hamburg.de}

\thanks{$^1$ Supported by Heisenberg-Stipendium BO 3699/1-2 of the DFG (German Research Foundation)}
\thanks{$^2$ Partially supported by the RTG 1670 of the  DFG (German Research Foundation)}

\begin{abstract}
We compute the dynamical degrees of certain compositions of reflections in points on a smooth cubic fourfold. Our interest in these computations stems from the irrationality problem for cubic fourfolds. Namely, we hope that they will provide numerical evidence for potential restrictions on tuples of dynamical degrees realisable on general cubic fourfolds which can be violated on the projective four-space. 
\end{abstract}

\maketitle

\section{Introduction}\xlabel{sIntroduction}

Let $Y$ be a smooth complex projective $n$-fold and $f\colon Y \dasharrow Y$ be a birational self-map. Given such an $f$, one can associate a tuple of real numbers $\lambda_i(f)$, $0\leq i\leq n$, $\lambda_i(f) \ge 1$, $\lambda_0(f) = \lambda_n (f) =1$, with it. These numbers, the \emph{dynamical degrees}, measure the dynamical complexity of $f$. Letting $f$ run through the group of birational transformations of $Y$ gives the dynamical spectrum $\Lambda(Y)$ which is a birational invariant. 

We are interested in how properties of $\Lambda (Y)$ may reflect geometric properties of $Y$. In particular, it is an interesting question whether the dynamical spectrum may even carry enough information to distinguish between (conjecturally irrational) very general smooth cubic hypersurfaces $X \subset \PP^{n+1}$ and projective $n$-space $\PP^n$ itself (for $n\geq 4$). The spectrum does ``see'' rationality in dimension $2$: If $S$ is an irrational surface, then $\Lambda(S)$ is discrete, whereas for rational $S$ the spectrum has accumulation points from below and, in fact, infinite Cantor-Bendixson rank (the accumulation points accumulate again, and so forth ad infinitum). Another instance which shows that $\Lambda$ is closely connected with rationality is the following result:
If for a projective $n$-fold $Y$ there exists a birational self-map $f$ with a tuple of dynamical degrees such that any two consecutive $\lambda_i(f)$ are distinct, then $Y$ has Kodaira dimension $0$ or $-\infty$, see \cite[Cor.\ 1.4]{DinhNgu11}.

From now on, let $X$ be a very general smooth cubic fourfold, and denote by $\mathrm{Bir}_0 (X) \subset \mathrm{Bir}(X)$ the subgroup generated by reflections $\sigma_p$ in points $p\in X$. In this article we will try to obtain some constraints for the tuples of dynamical degrees $(\lambda_i(g))$ for $g\in \mathrm{Bir}_0 (X)$. Our goal here is to compute the dynamical degrees of many elements in $\mathrm{Bir}_0 (X)$ to get a feeling as to which ones can be realized on a very general cubic. Our main results can be summarised as

\begin{theorem}\xlabel{tMAIN}
Let $X$ be a very general smooth cubic fourfold. 
\begin{itemize}
\item[(1)]
[See Theorem \ref{tGeneral}, (c)] 
For a general $N$-tuple of points on $X$, $N\ge 3$, the dynamical degrees of the composition of the reflections in these points are $(1,2^N,2^N,2^N,1)$. 
\item[(2)]
[See Theorem \ref{tConicLine}]
Consider a general line $L$ on $X$ and a general plane section of $X$ through $L$, which then decomposes as $L\cup C$, where $C$ is a conic. Pick general points $p,q\in L$ and $r\in C$. If  $g=\sigma_r\sigma_q\sigma_p$, then $\lambda_1(g)=\lambda_3(g)=\frac{5+\sqrt{33}}{2}$.
\item[(3)]
[See Theorem \ref{tTriangle}]
If $p, q,r$ are the vertices of a triangle of lines on $X$ and $g=\sigma_r\sigma_q\sigma_p$, then $\lambda_1 (g) = \lambda_3 (g) = \left( \frac{1+ \sqrt{5}}{2}\right)^3$.
\end{itemize}
\end{theorem}

We have not computed $\lambda_2$ in the last two cases of the Theorem yet, but see Remark \ref{rSecond} for some comments about the computations of $\lambda_2$ in case (2).

Let us now briefly describe a potential ``Ansatz" to prove irrationality of a very general cubic fourfold $X$: prove constraints for the dynamical degrees on $X$ and show that these constraints are violated on $\PP^4$ by constructing an example; note that for many examples of Cremona transformations, e.g. monomial ones \cite{Lin13}, the dynamical degrees are readily computable. These constraints could either be of an arithmetic nature, for example, which algebraic number fields the $\lambda_i$ lie in, or consist of new inequalities, for example, bounds for the size of the ratio $\max \{\log(\lambda_2)/\log(\lambda_1), \log(\lambda_2)/\log(\lambda_3)\}$.

Note, that there exist birational self-maps with interesting multi-degrees, but uninteresting dynamical degrees. Indeed, Pan in \cite{Pan00} and \cite{Pan13} constructed birational self-maps of $\PP^3$ of all bidegrees allowed by the Hodge and Cremona inequalities (see \cite[Prop.\ 7.1.7 \& Rem.\ 7.1.8]{Dolg12}) just by considering de Jonquieres maps, i.e. maps extended from $\PP^2$ and preserving a linear fibration on $\PP^3$. However, the dynamical degrees of these have $\lambda_1 = \lambda_2$ by \cite{DinhNgu11}. Therefore, to prove constraints one will presumably have to work with all iterates of a given birational map $f$ and not only a sufficiently high power of $f$. This approach to distinguishing rational from irrational varieties ties in well with the old philosophy that varieties closer to rational ones should admit more or ``wilder'' birational self-maps; this goes back to the work on birational rigidity of Iskovskikh-Manin \cite{I-M71}, see also the book by Pukhlikov \cite{Pukh13} for lots of developments in this direction. The main new point of view here is that we propose the dynamical spectrum as a means to \emph{measure the size} of $\mathrm{Bir}(X)$, or in other words, to quantify its ``wildness". This means we want to use \emph{quantitative differences} (the dynamical spectrum) instead of only qualitative ones. Compare this with the fact that in the recent article \cite{B-L14} the authors prove that $\mathrm{Bir}(\PP^3)$ and $\mathrm{Bir}(Y)$, where $Y$ is  a cubic threefold, are very much alike in several qualitative aspects. For instance, in both one can find birational self-maps contracting surfaces birational to any given birational type of a ruled surface. But there are examples of two varieties where in each case the birational automorphism group contains elements contracting subvarieties of uncountably many birational types, but whose dynamical spectra can be seen to be distinct. For instance, one can take $\PP^{n+2}$ and $Z=\PP^n \times Y$, $n\ge 3$, $Y$ a surface of general type without rational curves, hence no non-constant maps $\PP^n \dasharrow Y$.  Since every rational map $\PP^n \dasharrow Y$ is constant, every map in $\mathrm{Bir}(Z)$ preserves the fibration $Z\to Y$, hence cannot have pairwise different consecutive dynamical degrees \cite{DinhNgu11}. But $\PP^{n+2}$ does have maps with this property.

\smallskip

\noindent
\textbf{Conventions.} We work over the field of complex numbers $\CC$ unless stated otherwise. By a variety we mean a possibly reducible integral separated scheme of finite type over $\CC$. By a subvariety we mean a closed subvariety unless stated otherwise. If $f\colon X\dasharrow Y$ is a rational map, we denote by $\mathrm{dom}(f)$ the largest open subset of $X$ on which $f$ is a morphism. The graph $\Gamma_f \subset X \times Y$ of $f$ is the closure of the locus of points $(x, f(x))$ with $x \in \mathrm{dom}(f)$.

\section{Preliminaries}\xlabel{sPreliminaries}

 In his foundational paper on birational correspondences \cite{Zar43}, Zariski introduced the notations $f[Z]$ resp.\ $f\{Z\}$ for an irreducible subvariety $Z\subset X$, which he called, respectively, the \emph{(birational) transform} $f[Z]$ and \emph{total transform} $f\{Z\}$ (\cite[pp.\ 519--520]{Zar43}). Namely, 
the birational map $f$ induces an automorphism of function fields
\[
\xymatrix{
\CC (X) & \CC (X) \ar[l]^{f^*}_{\simeq}
}
\]
and one says that irreducible subvarieties $W$, $W'$ of $X$ \emph{correspond} to each other under $f$ if there is a (general) valuation $v\colon \CC (X)^* \to \Gamma$ (possibly not necessarily divisorial or of rank $1$), for some (totally ordered abelian) value group $\Gamma$, such that the center of $v$ on $X$ is $W$, and the center of $v\circ f^*$ on $X$ is $W'$. 

Then $f[Z]$ is the (possibly reducible) subvariety of $X$ such that (1) each irreducible component of $f[Z]$ corresponds to $Z$ and  (2) every irreducible subvariety of $X$ which corresponds to $Z$ is contained in $f[Z]$. On the other hand, $f\{ Z\}$ is the locus of all points on $X$ which correspond to some point in $Z$. We extend the definitions of $f[Z]$ and $f\{Z\}$ to reducible $Z$'s componentwise.

For example, let $f$ be an ordinary quadratic plane Cremona transformation. Then,  if $p$ is a base point of $f$, we have  $f[p] = L$, where $L$ is the line corresponding to the base point $p$ in the image. For a different example,  if $Z$ is a curve passing through a base point $p$ and not a component of the triangle of lines, then  $f[ Z ] $ is a curve and $f\{ Z \} = f [Z] \cup L$. 
In this example, $L$ corresponds to $p\in Z$, but $L$ does not correspond to $Z$!

Geometrically, $f[Z]$ resp. $f\{Z\}$ have the following meaning: 
Consider the graph $\Gamma_f \subset X \times X$ with its two projections $p_1 \colon \Gamma_f \to X$, $p_2 \colon \Gamma_f \to X$: 
\[
\xymatrix{
& \Gamma_f\ar[ld]_{p_1}\ar[rd]^{p_2} & \\
 X \ar@{-->}[rr]^f  &  & X.
}
\]
Then $f[Z]$ consists of the images on $X$ via $p_2$ of all the irreducible subvarieties of $\Gamma_f$ which map \emph{onto} $Z$ via $p_1$; $f\{ Z\}$ consists of the images on $X$ via $p_2$ of all the irreducible subvarieties of $\Gamma_f$ which map \emph{into} $Z$ via $p_1$.
\smallskip

\begin{definition}\xlabel{dMoves}
If a birational map $f \colon X\dasharrow X$ is a composition of birational maps 
\[
\xymatrix{
X \ar@{-->}[r]^{f_0} & X \ar@{-->}[r]^{f_2} & X  \ar@{-->}[r]^{f_3}  &  \ldots  \ar@{-->}[r]^{f_{N-1}} & X  \ar@{-->}[r]^{f_N} & X 
}
\]
we define, for every $i\in \ZZ$, $X_i := X$, and $f_i\colon X_i \to X_{i+1}$ to be equal to the map $f_j$, $j\in\{0, \dots , N\}$, with $j\equiv i$ (mod $(N+1)$). 

Let $Z \in X_i$ be a subvariety. We define its \emph{move from position $i$ to position $j$} as the subvariety of $X_j$ given by
\begin{align*}
M_{i\shortrightarrow j} [Z]& := ( f_{j-1}\circ \dots \circ f_i) [ Z ] \quad\;\; \mathrm{if}\; j\ge i \\ 
M_{i\shortrightarrow j} [Z] &:= ( f_{j}^{-1}\circ \dots \circ f_{i-1}^{-1}) [ Z ] \quad \mathrm{if}\; j < i . 
\end{align*}


Of course, it makes sense to make the above definitions also with the brackets $[ \cdot ]$ replaced by $\{ \cdot \}$ everywhere, and then speak of the \emph{total move} etc.
\end{definition}

We now move on to the main subject of the article, the dynamical degrees. Let $Y$ be a smooth projective $n$-fold, $f\colon Y\dasharrow Y$ be a birational map. The induced map $f^{*}$  on the group of cycles $A^i(Y)$ of codimension $i$ is defined as follows: Choose a smooth resolution of $f$
\[
\xymatrix{
 & Z\ar[ld]_{\pi_1}\ar[rd]^{\pi_2} & \\
Y \ar@{-->}[rr]^f & & Y
}
\]
and define $f^* = \pi_{1*}\circ \pi_2^*$.

\begin{definition}
Let $Y$ and $f$ be as above and let $\rho_i$ be the spectral radius of $f^*$ on $A^i (Y)$. The \emph{$i$-th dynamical degree} of $f$ is 
\[\lambda_i (f) = \liminf\limits_{m\to \infty} (\rho_i ((f^m)^*))^{\frac{1}{m}}.\] 
\end{definition}

The $i$-th dynamical degree can also be written as 
\[\lambda_i(f)=\liminf\limits_{m\to \infty}(H^{n-i}.(f^m)^*H^i)^{\frac{1}{m}}\]
for any ample divisor $H$; see \cite[Thm.\ 2.4]{Guedj10}.  

The $\lambda_i(f)$ are invariant under birational conjugacy, see \cite[Cor.\ 2.7]{Guedj10}, are related to the entropy of $f$ in a precise sense (for example, the entropy $h(f)$ is bounded above by $\mathrm{max}_i \{ \mathrm{log} (\lambda_i(f)\}$, see \cite{DS05}), and one way to think of $\lambda_i (f)$ intuitively may be as the entropy of $f$ on algebraic cycles of codimension $i$, i.e.\ a measure of how much information the action of the iterates of $f$ on cycles of codimension $i$ carries.

\begin{definition}
The \emph{dynamical spectrum} of $Y$ is defined as
\[
\Lambda (Y) := \{ (\lambda_0, \dots, \lambda_n) \in \RR^{n+1} \mid f \in \mathrm{Bir}(Y)\}.
\] 
\end{definition}

Clearly, $\Lambda(Y)$ is a birational invariant of $Y$ and as a subset of $\RR^{n+1}$ it comes with interesting point-set, metric or topological properties.

The sequence $\lambda_0(f), \lambda_1 (f), \ldots , \lambda_{n-1}(f), \lambda_n(f)$ is known to be always \emph{log-concave}, i.e.\ $p \mapsto \mathrm{log} (\lambda_p(f))$ is concave, see \cite{DinhNgu11} and references therein; it implies that the sequence first strictly ascends for a while, then stays constant, then strictly descends again:
\[
1=\lambda_0(f) < \dots < \lambda_p(f) = \dots = \lambda_{p'} (f) > \lambda_{p'+1}(f) > \dots > \lambda_{n}(f) =1.
\]

Another important property we will frequently use is
\begin{equation}\label{log-concavity}
\lambda_j^2 \ge \lambda_{j-1}\lambda_{j+1}\; \forall j;
\end{equation}
see \cite[Thm.\ 2.4]{Guedj10}.

\begin{definition}\xlabel{dALgStab}
A birational map $f\colon Y\dasharrow Y$ is said to be \emph{algebraically $i$-stable} if
we have
\[
(f^m)^* = (f^*)^m \quad \forall m \in \NN,
\]
on $A^i(Y)$.
\end{definition}

\begin{lemma}\xlabel{lDivisors}
Suppose that neither $f \colon Y \dasharrow Y$ nor any iterate $f^j$ contract any divisors into the indeterminacy locus of $f$. Then $f$ is algebraically $1$-stable. The converse also holds. 
\end{lemma}

\begin{proof}
Let $I_f \subset Y$ be the indeterminacy set of $f$. Since $Y$ is smooth, $I_f$ has codimension at least $2$, and all preimages \[ \overline{( f^{k}\mid_{\mathrm{dom}(f^k)})^{-1}(I_f)}, \quad k\ge 1, \] have codimension at least $2$ since no divisor is contracted into $I_f$ under any iterate of $f$. Hence, for any divisor class $D$, we have that $f^* (f^* D)$ and $(f^2)^* (D)$ 
coincide in codimension $1$, hence everywhere, and the same argument yields $(f^*)^n = (f^n)^*$.

Conversely, let us assume that some iterate of $f$, without loss of generality $f$ itself, contracts a divisor into the indeterminacy locus $I_f$ of $f$. Then $f^* (f^* H)$, where $H$ is the class of any ample divisor,  and $(f^2)^* (H)$ differ by this  effective cycle contracted by $f$. 
\end{proof}

Furthermore, we will use

\begin{lemma}\xlabel{lInverse}
Let $Y$ and $f \colon Y \dasharrow Y$ be as above. The dynamical degrees satisfy the equation
\[
\lambda_{n-i}(f^{-1}) = \lambda_i (f).
\]
\end{lemma}

\begin{proof}
Let $H$ be a general hyperplane section of $Y$. Then the map $f$ (and each iterate of it) has a multidegree $(d_0 (f), \ldots , d_n(f))$ where $d_i$ is the degree, relative to $H^{n-i}$, of the birational transform of $H^i$. We will show that even
\[
d_{n-i}(f) = d_i (f^{-1}). 
\]
This follows immediately from the fact that $\Gamma_{f^{-1}} = \widetilde{\Gamma}_f$ where $\Gamma_{f^{-1}} \subset Y \times Y$ is the graph of $f^{-1}$ and $ \widetilde{\Gamma}_f$ is the image of the graph of $f$ under the map that switches the factors in $Y\times Y$. 
\end{proof}

We also recall the following special case of the main result in \cite{DinhNgu11}. If $X\subset \PP^{n+1}$ is an $n$-dimensional (smooth) cubic hypersurface, $l\subset X$ a line, and 
\[
\pi_l\colon X_l = \mathrm{Bl}_l (X) \to \PP^{n-1}
\]
the induced conic fibration, and if, moreover, $f\in \mathrm{Bir} (X)$ preserves this fibration, and $\overline{f}\in \mathrm{Bir}(\PP^{n-1})$ is the induced map, then \cite[Thm.\ 1.1]{DinhNgu11} implies 
\begin{equation}\label{degrees-fibration}
\lambda_k (f) = \mathrm{max} \{ \lambda_k (\overline{f}), \lambda_{k-1} (\overline{f})\} \quad \forall k. 
\end{equation}

\section{Geometry of a single reflection}\xlabel{sSingleReflection}

We fix a smooth cubic fourfold $X$ and a very general point $p\in X$, and want to describe in more detail the geometry of the birational reflection map $\sigma_p$ and of its resolution. Moreover, we want to understand the geometry of the indeterminacy locus and its resolution.
  
There is a commutative diagram
\[
\xymatrix{
\widetilde{X} \ar[d]_{\mathrm{bl}_{S'(p)}} \ar[rdd]^{\varphi}  &  \\
X'  \ar[d]_{\mathrm{bl}_p}  &    \\
X \ar@{-->}[r]^{\sigma_p} &  X.
}
\]

Here $\mathrm{bl}_p$ is the blow-up of the point $p\in X$. Its exceptional divisor on $X'$ is denoted by $E'(p)$. Let 
\[
Y (p) := X \cap \mathbb{T}_p X 
\]
be the intersection of $X$ with the embedded projective tangent space $\mathbb{T}_p X \simeq \PP^4$ to $X$ in $p$ inside $\PP^5$. Then $Y(p)$ is a singular cubic threefold with a single node at $p$ for a generic choice of $p$, which was our standing assumption. Its strict transform inside $X'$ is denoted by $Y'(p)$. In $Y'(p)$, the singular point $p$ has been replaced by the projectivized tangent cone of $Y(p)$ in $p$, which is $\PP^1\times \PP^1$. 

\begin{remark}\xlabel{rIndeterminacylocus}
The indeterminacy locus of $\sigma_p$ is the locus of all lines through $p$ inside $X$. 
\end{remark}

Note that all the lines through $p$ inside $X$ form a surface $S(p)$ inside $Y(p)$ which is a cone over a curve $C(p)$ of bidegree $(3,3)$ in $\PP^1 \times \PP^1$ where we view this $\PP^1 \times \PP^1$ here as the intersection of the tangent cone of $Y(p)$ in $p$ with a $\PP^3$ inside $\mathbb{T}_p X$ which does not pass through $p$. 
Namely, by \cite[p.\ 187]{FW89}, the equation of a nodal cubic threefold in $\PP^4$ (in our case $Y(p)$) can be written as
\[
\xi_4Q+R=0
\]
in appropriate homogeneous coordinates $(\xi_0:\dots : \xi_4)$ with $Q$ resp.\ $R$ in $\CC[\xi_0, \dots , \xi_3]_2$ resp.\ $\CC[\xi_0, \dots , \xi_3]_3$ and the node equal to the point $(0:0:0:0:1)$. Then a line through $p$ can be written in parameter form as
\[
(0:0:0:0:1) + \lambda (a:b:c:d:0) .
\]
This lies on the nodal cubic if and only if $Q(a:b:c:d)=R(a:b:c:d)=0$.

Generically, $C(p)$ will be smooth, but of course all sorts of singular and/or reducible curves can occur if we drop the generality assumption; the possibilities are listed and discussed in \cite{FW89}, see also \cite{W87}.

The strict transform $S'(p)$ of $S(p)$ on $X'$ is a (smooth) ruled surface over $C(p)$. In a second step, we blow up $S'(p)$ inside $X'$ to obtain $\widetilde{X}$, thus we replace it by the projectivization of its normal bundle inside $X'$. The strict transforms of $E'(p)$, $Y'(p)$ on $\widetilde{X}$ are denoted by $\widetilde{E}(p)$, $\widetilde{Y}(p)$.  One has that $\widetilde{E}(p)$ is the blow-up of $E'(p) \simeq \PP^3$ in the curve $C(p)$. Let $\widetilde{F}(p)$ be the exceptional divisor of $\mathrm{bl}_{S'(p)}$. It is a $\PP^1\times \PP^1$-bundle over the curve $C(p)$. 

The geometry of the resolution is summarized in Figure \ref{figure1} below.

\begin{figure}
\caption{Geometry of the resolution of a reflection.}\label{figure1}
\begin{center}
\includegraphics[height=150mm]{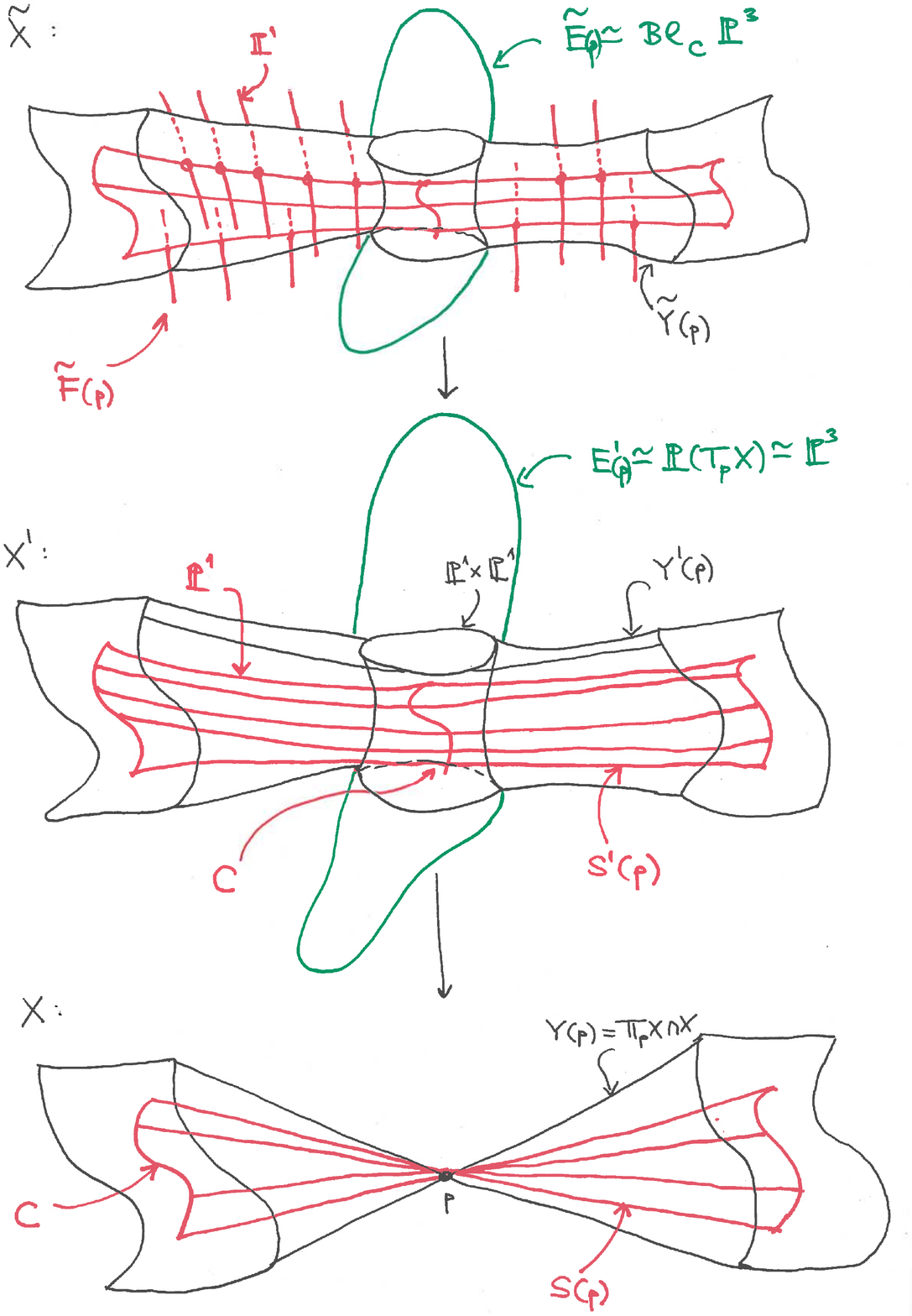}
\end{center}
\end{figure}

One has that $\widetilde{E}(p)$ and $\widetilde{F}(p)$ intersect in the exceptional divisor of the blow-up of $\PP^3$ in $C(p)$, a ruled surface over $C(p)$; $\widetilde{E}(p)$ and $\widetilde{Y}(p)$ intersect in a $\PP^1 \times \PP^1$; and $\widetilde{F}(p)$ and $\widetilde{Y}(p)$ intersect in a ruled surface isomorphic to $S' (p)$. 

Notice that in this case the birational self-map $\sigma_p$ lifts to an automorphism $\widetilde{\sigma}_p$:
\[
\xymatrix{
\widetilde{X}\ar[d] \ar[r]^{\widetilde{\sigma }_p} & \widetilde{X} \ar[d]  \\
X \ar@{-->}[r]_{\sigma_p}  &   X.
}
\]

Let $\widetilde{H}$ be the strict transform of a general hyperplane section $H$ of $X$. 

\begin{proposition}\xlabel{pLinearSystemReflection}
The morphism $\varphi$ is given by the linear system 
\[
 | 2\widetilde{H} - 3\widetilde{E}(p) - \widetilde{F}(p)|.
 \]
\end{proposition}

\begin{proof}
We use a direct computation, building on the proof of \cite[Prop.\ 12.13]{Manin74}: we choose homogeneous coordinates $X_0, \ldots, X_{5}$ in $\PP^{5}$ such that $p= (1: 0: 0: \ldots : 0)$ and $X_1 =0$ is the equation of the projective embedded tangent hyperplane to $X$ at $p$. The equation of $X$ can then be written as
\[
X_1X_0^2 + X_0 q(X_1, \ldots, X_{5}) + c(X_1, \ldots, X_{5}) = 0,
\]
where $q$ is a homogeneous quadratic form and $c$ a homogeneous cubic form in the variables $X_1, \ldots , X_5$. Then the reflection $\sigma_p$ can be described as
\[
\sigma_p (X_0 : \ldots : X_{5}) = (X_0X_1 + q (X_1, \ldots , X_{5}), -X_1^2, -X_1X_2, \ldots , -X_1X_{5}).
\]
Here, as was said above, $\{ X_1 =0 \} = \mathbb{T}_pX$, and $\{ X_1=0, q=0 \} \subset \mathbb{T}_pX\simeq \PP^4$ defines the tangent cone at $p$ to the singular nodal cubic threefold $X \cap \mathbb{T}_pX$. It is a cone over the quadric $\{ X_1=0, X_0 =0, q=0 \} \simeq \PP^1\times \PP^1$ in the hyperplane $\PP^3_{\infty} = \{ X_1=0, X_0=0 \} \subset \PP^4$ at infinity. The curve $C(p) \subset \PP^1 \times \PP^1$ corresponding to lines contained in $X\cap \mathbb{T}_pX$ (resp.\ $X$) is defined by $\{ X_1=0, X_0 =0, q=0, c=0 \}$. It is of bidegree $(3,3)$ in $\PP^1 \times \PP^1$.

Now let $(1: \epsilon_1 : \ldots : \epsilon_5 )$ be a jet of order $k$ centered at $p$, i.e. an element of $\hat{\mathcal{O}}_{\PP^5, p}/ ($polynomials of degree $> k$ in local  coordinates $x_i=X_i/X_0$ centered at $p$). The residues classes of the $x_i$ are the $\epsilon_i$. The condition that a $1$-jet is contained in $X$ is $\epsilon_1=0$, and the condition that a $2$-jet is contained in $X$ is 
\[
\epsilon_1 =0,\;  q(0, \epsilon_2, \ldots , \epsilon_5 ) =0 .
\]
Hence $\sigma_p$ vanishes on every $2$-jet contained in $X$ and centered at $p$. This means that it is defined by a linear subsystem of $| 2\widetilde{H} - 3\widetilde{E}(p)|$ on $\widetilde{X}$. Moreover, clearly ${\sigma}_p$ is undefined on the surface $\{ X_1=0, q=0, c=0 \} = S(p)$ inside $X$ as well, so that $\widetilde{\sigma}_p$ is given by a linear subsystem of $|2 \widetilde{H} - 3\widetilde{E}(p) - \widetilde{F}|$. It is easy to check directly that the quadrics in the above formula for $\sigma_p$ generate the space of all quadrics on $X$ which vanish along $S(p)$ and contain all $2$-jets centered at $p$. 
\end{proof}

More generally, we have that $\widetilde{H}, \widetilde{E}(p), \widetilde{F}(p)$ is a basis of $\mathrm{Pic}(\widetilde{X})$, $\widetilde{Y}(p) \equiv \widetilde{H} - 2\widetilde{E}(p)-\widetilde{F}(p)$, and the automorphism induced by $\widetilde{\sigma}_p$ on $\mathrm{Pic}(\widetilde{X})$ can be represented in the preceding basis by the matrix

\[
\left( \begin{array}{ccc}
2 & 1 & 0 \\
-3 & -2 & 0\\
-1 & -1 & 1
\end{array}\right).
\]

\section{First examples of dynamical degrees of compositions of reflections}\xlabel{sVeryGeneral}

A useful metaphor for our study of the dynamics of composites of reflections on a smooth cubic fourfold $X$ may be  the subject of billiards, see e.g. \cite[Ch.\ 9]{KaHa95}. There generic orbits often display some form of ergodicity and confirm to a uniform pattern, whereas special orbits, e.g. periodic ones such as star-shaped closed inscribed polygons for circular billiards, may be interesting but harder to make general assertions about. Therefore, we first examine composites of reflections in very general collections of points in $X$ and then afterwards permit the points to attain some more special geometric configurations; however, only those configurations are of interest to us which are realizable on a \emph{very general cubic}, since this is the situation where we want to get a feeling for which tuples of dynamical degrees $(\lambda_1, \lambda_2, \lambda_3)\in \RR^3$ can occur.

\begin{theorem}\xlabel{tGeneral}
Let $N$ be a positive integer and let $\underline{p}=(p_1, \ldots , p_N) \in X^N$ be a very general $N$-tuple of points on a smooth cubic fourfold $X$. Let 
\[
\sigma_{\underline{p}} = \sigma_{p_N} \circ \ldots \circ \sigma_{p_1}
\]
be the associated composition of reflections, and \[ \lambda_{\underline{p}} = (\lambda_1 (\sigma_{\underline{p}} ), \lambda_2 (\sigma_{\underline{p}} ), \lambda_3 (\sigma_{\underline{p}} ) )\]
the associated triple of dynamical degrees (note that, clearly, $\lambda_0 (\sigma_{\underline{p}} )= \lambda_4 (\sigma_{\underline{p}} )=1$ are not interesting). Then the following holds:
\begin{itemize}
\item[(a)]
For all $N$, $\lambda_{\underline{p}}$ does not depend on $\underline{p}$, but only on $N$.
\item[(b)]
If $N=1$, we have \[ \lambda_{\underline{p}} = (1,1,1) \]
and for $N=2$ we also get
\[
 \lambda_{\underline{p}} = (1,1,1) .
\]
\item[(c)]
For $N \ge 3$ we have \[ \lambda_{\underline{p}} = (2^N, 2^N, 2^N) . \]
\end{itemize}
\end{theorem}

\begin{proof}
Clearly, for $N=1$, $\lambda_{\underline{p}} = (1,1,1)$, since $\sigma_{p_1}$ is a single reflection and therefore is of finite order. 

If $N=2$, then we can assume that $p=p_1$ and $q=p_2$ do not lie on a line which is contained in $X$ since we assumed the tuple of points to be a very general one.  Let $r=p_3$ be the third intersection point of $\overline{p_1p_2}$ with $X$. Then blowing up $p_1, p_2, p_3$ yields a model $\widetilde X$ on which $\sigma_{\underline{p}}$ is algebraically $1$-stable by Lemma \ref{lDivisors}. Indeed, there are only six divisors on this blow-up which a priori could be contracted into the indeterminacy locus, namely the three exceptional ones and the three strict transforms of the tangent hyperplane sections in the points. But note that the lift of every $\sigma_{p_i}$ is defined in the generic point of each of these six divisors and permutes them. 

Let $H$ be the pull-back of a hyperplane section of $X$ to this blow-up $\widetilde{X}$, and let $P, Q, R$ be the corresponding exceptional divisors (isomorphic to $\PP^3$) lying over $p$, $q$, $r$ respectively. In the ordered basis $(H, P, Q, R)$ of $\mathrm{Pic}(\widetilde{X})$ the matrix of $\sigma_q \circ \sigma_p$ is equal to

\begin{gather*}
\begin{pmatrix}
4 & 2 &  0 & 1 \\
0 & 0 & 1 & 0 \\
-6 & -3 & 0  & -2 \\
-3 & -2 & 0 & 0
\end{pmatrix}.
\end{gather*}

Indeed, first of all note that $Q$ is first mapped to $R$ under $\sigma_p$, then onto $P$ under $\sigma_q$ (strictly speaking, we, of course, are talking about the lifts of the reflections). Similarly, $R$ is mapped onto $Q$ under $\sigma_p$ and so forth. This argument gives the third column of the matrix.

Now note that $H$ is transformed under $\sigma_p$ into $2H-3P$ (this follows as in the proof of Proposition \ref{pLinearSystemReflection}), and similarly for $\sigma_q$. Hence $\sigma_q\sigma_p(H)$ is just $2(2H-3Q)- 3R$. 

The exceptional divisor $P$ is transformed under $\sigma_p$ into the divisor which is the strict transform of the cubic $Y(p) = X \cap \mathbb{T}_p X$ on $\widetilde{X}$ (recall that the node at $p$ gets resolved by replacing it by a $\PP^1\times \PP^1$). On $\widetilde{X}$, the latter strict transform is equivalent to $H - 2P$, which under $\sigma_q$ is transformed into $2H- 3Q -2R$. 

Lastly, $R$ maps first to $Q$ under $\sigma_p$, then $\sigma_q$ maps $Q$ onto the strict transform of $Y(q)$ which is equivalent to $H - 2Q$. 

Combining all these arguments gives the above matrix, whose characteristic polynomial is $(x-1)^4$, hence $\lambda_1=1$ and, therefore, all $\lambda_i$ are equal to $1$ by Equation \ref{log-concavity}.

We now turn to the general case $N\ge 3$. 

Notice that if $X$ itself is a model on which $\sigma_{\underline{p}}$ is $1$-stable, then the same will hold for all $\underline{q}$ in the complement of countably many proper subvarieties of $X^N$, i.e. for a very general $\underline{q}$: by Lemma \ref{lDivisors}, $1$-stability can be characterized geometrically by requiring that no iterate of $f$ contracts a divisor into the indeterminacy locus of $f$, and the contrary case can be expressed in terms of countably many algebraic equations for $\underline{q}$: this follows from Lemmas \ref{lAlgebraic1} and \ref{lAlgebraic2} below. Thus, in this case we get $\sigma_{\underline{q}}H=2^NH$ (compare the arguments above) for all such $\underline{q}$ and, therefore,
\[
\lambda_1 (\sigma_{\underline{q}}) = 2^N.
\]
Let $U \subset X^N$ be the complement of the countably many subvarieties one has to remove to characterize the set of $\underline{q}$; let $U'$ be the image of $U$ under the map reversing the factors in $X^N$, i.e. $(x_1, \ldots , x_N) \mapsto (x_N, \ldots , x_1)$. Then for all $\underline{q}'\in U'$
\[
\lambda_3 (\sigma_{\underline{q}'}) = 2^N
\]
since the $\sigma_{\underline{q}'}$ are nothing but the inverses of all the possible $\sigma_{\underline{q}}$. Here we used Lemma \ref{lInverse}. Since for very general $\underline{p}$ we have $U\cap U' \neq \emptyset$, it follows that
\[
\lambda_1 (\sigma_{\underline{p}}) = \lambda_3 (\sigma_{\underline{p}}) = 2^N.
\]
By Equation \ref{log-concavity}, $\lambda_2 (\sigma_{\underline{p}}) \ge 2^N$, but on the other hand, also $\lambda_2 (\sigma_{\underline{p}}) \le 2^N$ since the second (Cremona-)degree of each $\sigma_{p_i}$ is equal to $2$. Here we use the submultiplicativity of the degree, see \cite[Prop.\ 2.6]{Guedj10}. 

Now take a tuple of points $p_1, \dots , p_N$ in a general plane section $E= \PP^2 \cap X$ which is a smooth elliptic curve. Let us show that for this $\underline{p}$ the variety $X$ is already a model for which $\sigma_{\underline{p}}$ is algebraically $1$-stable. Indeed, the indeterminacy locus of $\sigma_{\underline{p}}$ intersected with $E$ is nothing but $p_1, \dots , p_N$ since $E$ is smooth and contains no lines as components (recall that, by Remark \ref{rIndeterminacylocus}, the indeterminacy loci of the $\sigma_{p_i}$ are precisely the lines through $p_i$). Moreover, for general $E$ and $p_1, \dots , p_N$ on $E$, no composition $\sigma_{p_j}\circ \dots \circ \sigma_{p_{i+1}}\mid_E$, $j\ge i+1$, maps $p_i$ to $p_{j+1}$. For instance, this will hold if $p_{i+1}$ is not in the subgroup generated by $p_1, \dots ,p_i$ for all $i$, which can be proven as follows: Notice that for points $x,y\in E$
\[
\sigma_x (y) = -x-y.
\]
Therefore, given $p_1, \dots , p_N$, we have
\begin{eqnarray*}
\sigma_{p_2}(p_1) &= -p_1-p_2\\
\sigma_{p_3}(-p_1-p_2) &= p_1+p_2 -p_3\\
\dots
\end{eqnarray*}
and we will now check that for $N$ even, this sequence never returns to $p_1$, whereas for $N$ odd, it does, but always after the application of a $\sigma_{p_1}$. Of course, since the situation is symmetric, it is indeed sufficient to consider the case of $p_1$.

If $N$ is even, after application of the first $\sigma_1$, the coefficient in front of $p_1$ is zero and remains zero until we apply $\sigma_1$ for the second time. Then the coefficient is $-1$. After the next application of $\sigma_1$, it is $-2$, after that $-3$ etc. 

If $N$ is odd, a direct computation, paying special attention to signs, shows that $p_1$ is mapped back to itself for the first time after applying
\[
\sigma_1\circ \left( \sigma_{N}\circ \dots \circ\sigma_1 \right) \circ \left( \sigma_N\circ\dots \circ\sigma_2\right) 
\]
and that no element on the way equals one of the $p_i$. 

 This finishes the proof of Theorem \ref{tGeneral}.   
\end{proof}

\begin{lemma}\xlabel{lAlgebraic1}
Let $Y \subset \PP^n$ be a smooth projective variety and let 
\[
\underline{F} = (F_0 (x_0, \ldots, x_n) , \ldots , F_N (x_0, \ldots , x_n))
\]
be an $N+1$-tuple of homogeneous polynomials $F_i \in \CC [x_0, \ldots , x_n]_d$ of the same degree $d$ representing a birational map $f \colon Y \dasharrow Y$ (hence, in particular, the $F_i$ do not vanish simultaneously on $Y$). Then the subset of those $\underline{F} \in \PP (\CC [x_0, \ldots , x_n]_d^{\oplus{n+1}})$ giving rise to birational $f$'s and such that $F_0=\ldots =F_N=0$ does not contain a codimension one algebraic subset of $Y$ form a locally closed subvariety  $\mathcal{P}_d$ of $\PP (\CC [x_0, \ldots , x_n]_d^{\oplus{n+1}})$. Those $\underline{F}$ such that in addition an iterate of $f$ contracts a purely one codimensional algebraic subset $Z \subset Y$ into the indeterminacy locus $I_f$ of $f$ form a countable union $\mathcal{Z}$ of closed algebraic subsets of $\mathcal{P}_d$. 
\end{lemma}

\begin{proof}
The fact that the $\underline{F} \in \PP (\CC [x_0, \ldots , x_n]_d^{\oplus{n+1}})$ giving rise to birational $f$'s and such that $F_0=\dots =F_N=0$ does not contain a codimension one algebraic subset of $Y$ form a locally closed subvariety $\mathcal{P}_d$ of $\PP (\CC [x_0, \ldots , x_n]_d^{\oplus{n+1}})$ can be  proven analogously to \cite[Prop.\ 2.4 \& 2.5]{BBB14}.  

Let us show the second assertion. Let $H$ be one of the countably many components of finite type of the Hilbert scheme parametrizing purely one codimensional subschemes of $Y$, and let 
\[
\pi\colon \mathcal{H} \subset H \times \PP^N \to H
\]
be the universal family.
Let
\[
\mathcal{Y}_i =\left\{ (h,x,\underline{F}^{i+1}) \in \mathcal{H}\times\mathcal{P}_d\subset H\times\PP^n\times \mathcal{P}_d \mid \underline{F}^{i+1}(x)=0 \right\}
\]
and consider the natural projection
\[
\alpha\colon\mathcal{Y}_i \to H \times\mathcal{P}_d .
\]
By upper semi-continuity of fiber dimension of $\alpha$, the set
\[
\mathcal{X}_i =\{ z=(h, \underline{F}) \subset H \times \mathcal{P}_d \mid \dim \alpha^{-1}(z)\ge \dim X -1\} 
\]
is closed in $H \times \mathcal{P}_d$ (note that we use that the $F_j$ do not vanish on a common codimension one subset). The projection $p\colon H \times \mathcal{P}_d  \to \mathcal{P}_d$ is proper, hence $p(\mathcal{X}_i)\subset \mathcal{P}_d$ is closed. Taking the union over all $i$ and the countably many components of the Hilbert scheme gives us the description of the subset $\mathcal{Z}$ of $\mathcal{P}_d$ as claimed. 
\end{proof}

\begin{lemma}\xlabel{lAlgebraic2}
Let $X$ be a smooth cubic fourfold, $p_1, \ldots , p_N$ a tuple of points in $X$ as before. Then, using the notation of the preceding lemma, there is a $d\in \NN$, an open neighborhood $\Omega$ of $(p_1, \ldots , p_N)$ in $X^N$ and a morphism
\[
s\colon \Omega \to \mathcal{P}_d
\]
such that $s(q_1, \ldots , q_N)$ represents the composite of reflections $\sigma_{\underline{q}} = \sigma_{q_N} \circ \ldots \circ \sigma_{q_1}$. 
\end{lemma}

\begin{proof}
It suffices to do the proof for a single point $p_1$ and then it is a direct calculation as in the proof of Proposition \ref{pLinearSystemReflection}. 
\end{proof}

We retain the notation of the preceding section and, in particular, of Theorem \ref{tGeneral}. 
Having settled the generic situation, we pass on to some more special configurations:

\begin{proposition}\xlabel{pPointsOnLines}
Let $p_1, p_2$ be points in $X$ such that $\overline{p_1p_2}$ is contained in $X$. Then 
\[
\lambda_{\underline{p}} = (1,1,1).
\]
\end{proposition}

\begin{proof}
It suffices to notice that $\sigma_{\underline{p}}$ is a lift of the identity map on $\PP^3$ along the conic fibration $X \dasharrow \PP^3$ given by projecting from $\overline{p_1p_2}$. Hence, the assertion follows by Equation \ref{degrees-fibration}.
\end{proof}

\section{A conic and a line}\xlabel{sConicLine}

We now begin with a discussion of a special configuration of points and the computation of dynamical degrees in this case. It is one of the first instances where more interesting dynamics arises. 

\subsection{The computational approach}\xlabel{ssComputation}

Choose a plane $\PP^2$ in such a way that 
\[X\cap \PP^2=L\cup C,\]
where $L$ is a line and $C$ is a conic intersecting the line transversely. Let $p, q, r$ be three points on $L\cup C$ with $p,q\in L$, $r\in C$. Let $a, b$ be the two intersection points of $C$ and $L$, and assume that neither of them coincides with $p,q$ or $r$. The rough picture is shown in Figure 2.
\setlength{\unitlength}{1cm}

\begin{figure}\label{fBasicConfiguration}
\caption{}
\begin{picture}(3,3)
\put(-1,2){$\bullet$}
\put(-1,1.6){$p$}

\put(4,2){$\bullet$}
\put(4,1.6){$q$}

\put(-2,2.1){\line(1,0){8}}
\put(2.6,2.2){$L$}

\put(1.4,1){$\bullet$}
\put(1.4, 0.6){$r$}

\put(1,1.6){\circle{6}}
\put(0,1){$C$}
\end{picture}
\end{figure}

We want to compute the first and third dynamical degrees of 
\[
g = \sigma_r\circ\sigma_q\circ\sigma_p .
\]

We introduce some notation useful for the sequel: We will write $p_0=p, p_1=q, p_2=r$ occasionally when convenient for indexing purposes, $\TT_i =\TT_{p_i}X \cap X$ for the tangent hyperplane sections, $\mathcal{L}_i$ for the surface of lines on $X$ through $p_i$. Note that in the notation of Section \ref{sSingleReflection}, $\mathbb{T}_i= Y(p_i)$ and $\mathcal{L}_i = S(p_i)$.

\smallskip

Since computing $\lambda_1 (g)$ and $\lambda_3 (g)$ even in this at first glance comparatively harmless case involves a lot of technical details and auxiliary considerations, let us outline first of all the general method we will be pursuing. We will compute $\lambda_3 (g)$ and then deduce $\lambda_1 (g)$ by using the symmetries of the situation.

\noindent \textbf{Step 1.} We start with a very general curve $\Gamma\subset X$ which is the intersection of three members of a very ample linear system on $X$. Now 
\[
\lambda_3 (g) = \liminf_{n\to\infty} (\mathrm{deg} ((g^n)^* (\Gamma ))^{\frac{1}{n}}
\]
where $\mathrm{deg} ((g^n)^* (\Gamma )$ denotes the degree of the birational transforms of $\Gamma$ under $g^n$ with respect to the chosen very ample linear system. Our approach is elementary inasmuch as it aims at computing the degrees of these birational transforms directly, and then we will determine their exponential growth rate, which gives $\lambda_3 (g)$, after that.

However, for these computations to work, we need several genericity assumptions to hold for $\Gamma$, and the hardest part of the computation consists in showing that the set of $\Gamma$ satisfying all of them is actually not-empty. 

\smallskip

\noindent \textbf{Step 2.} Let us explain how we will compute $x_{\nu}(1):=d_{\nu}:=\deg (g^{\nu})^* (\Gamma)$ for $0 \le \nu < \infty$ and determine the asymptotic growth rate of them. Apart from the degrees of the birational transforms of $\Gamma$ we will also consider some auxiliary integers $x_{\nu} (2), \dots , x_{\nu} (r)$ that capture the salient features of the state of the discrete dynamical system generated by $g$ at time $\nu$, starting from a $\Gamma \in (H)^3$, sufficiently well so as to determine the set of integers
\[
(x_{\nu +1} (1), \dots , x_{\nu +1} (r))
\]
for the next time moment $\nu+1$. For example, the $x_{\nu} (2), \dots , x_{\nu} (r)$ may encode some multiplicities, number of certain points lying on distinguished loci, etc., at time $\nu$. The main point is that, if we introduce an integer vector
\[
v_{\nu} := (x_{\nu } (1), \dots , x_{\nu } (r))^t \in \ZZ^r, 
\]
then the transition from one state of the system to the next will be affected by a linear transformation $A\colon \ZZ^r \to \ZZ^r$ 
\[
v_{\nu +1} = A v_{\nu} .
\]
Moreover, usually $x_0(2) = \dots = x_0(r) =0$, i.e.  we start with
\[
v_{0} := (d_{0}, 0,  \dots , 0)^t .
\]

\begin{lemma}\xlabel{lMatrixEigenvalues}
In the above set-up, suppose that $A$ has eigenvalues $\mu_1, \dots , \mu_r$ with a multiplicity one positive real eigenvalue $\mu_1$ of maximum absolute value, and assume the eigenvalues are  ordered such that 
\[
\mu_1=|\mu_1 | > | \mu_2 | \ge \dots \ge | \mu_r| .
\]
Suppose that $v_0$ or, equivalently, the vector $e_1=(1,0,0, \dots , 0)$ is not in the span of the eigenspaces for $\mu_2, \dots , \mu_r$, and that the eigenspace for $\mu_1$ is not in the span of the vectors $e_2, \dots , e_r$ of the standard basis. Then the third dynamical degree equals $\mu_1$. 
\end{lemma}

\begin{proof}
The third dynamical degree is the exponential growth rate in $n$ of the first entry in the vector
\[
A^n v_0 . 
\]
Let $B$ be the base change matrix from the standard basis to the eigenbasis of $A$. Then we can rewrite
\[
A^n v_0 = B^{-1}(BAB^{-1})^n (Bv_0)
\]
and, since $v_0$ is not in the span of the eigenspaces for $\mu_2, \dots , \mu_r$, $Bv_0 = (b_1 , \dots , b_r)^t \cdot d_0$ with $b_i\in \CC$, $b_1\neq 0$. Moreover, since 
\[
(BAB^{-1})^n = \mathrm{diag}( \mu_1^n, \mu_2^n, \dots , \mu_r^n)
\]
and the eigenspace for $\mu_1$ is not in the span of the vectors $e_2, \dots , e_r$ of the standard basis, in other words, the $(1,1)$-entry of $B^{-1}$ is nonzero, we get that the first entry in $A^n v_0$ can be written as
\[
(c_1 \mu_1^n + c_2 \mu_2^n + \dots + c_r \mu_r^n )\cdot d_0, \quad c_i \in \CC, c_1 \neq 0. 
\]
Hence
\begin{gather*}
\lim\limits_{n\to\infty} \left( (c_1 \mu_1^n + c_2 \mu_2^n + \dots + c_r \mu_r^n )\cdot d_0 \right)^{\frac{1}{n}} \\
= \mu_1  \lim\limits_{n\to\infty} \left( \left( c_1 + c_2 \left( \frac{\mu_2}{\mu_1}\right)^n + \dots + c_r \left( \frac{\mu_r}{\mu_1}\right)^n \right) \cdot d_0 \right)^{\frac{1}{n}} = \mu_1.
\end{gather*}
\end{proof}

\smallskip

\noindent \textbf{Step 3.} The genericity properties which we need to compute the degree of $(g^{\nu})^* (\Gamma)$, typically fall into two categories. Firstly, we need that for  very general choice of $X, L, C, p, q, r$, and $Z$ a tangent divisor $\mathbb{T}_i$, the backward in time move $M_{i \shortrightarrow k}[Z]$ does not coincide with a tangent divisor $\mathbb{T}_{k-1}$, for any $k < i$. Equivalently, $\mathbb{T}_{i}$ does not get contracted in the backward evolution of the discrete dynamical system. 
This is needed because in this way it becomes possible to phrase the genericity properties the birational transform $(g^{\nu})^* (\Gamma)$ must have with respect to $\mathbb{T}_{\nu}$ as properties that the initial curve $\Gamma$ must have with respect to $M_{\nu \shortrightarrow 0}[ \mathbb{T}_{\nu}]$. We emphasize that this amounts to having genericity properties of the given configuration we start from, i.e. the data $X, L, C, p, q, r$, and not the auxiliary $\Gamma$ we choose later.

\smallskip

\noindent \textbf{Step 4.} Certain branches of the birational transforms $(g^{\nu})^* (\Gamma)$ (we will make this precise only later below) must not pass through a distinguished point $p,q,r$ at any time $\nu$. It will turn out that this can be accomplished provided $\Gamma$ intersects the $M_{\nu \shortrightarrow 0}[ \mathbb{T}_{\nu}]$ sufficiently generically and provided the chosen configuration $X, L,C, p,q,r$ has some additional genericity properties. We prove that all these genericity properties can be satisfied, or, equivalently, that the corresponding countable intersections of Zariski open sets are non-empty. 

\subsection{Dynamics of tangent divisors}\xlabel{ssTangent1}

Our objective here is to show

\begin{proposition}\xlabel{pGenericConfiguration}
For a very general choice of $X, L, C, p,q,r$, the following holds: 
The subvariety $M_{i\shortrightarrow k}[\mathbb{T}_i]$ is not contained in $\mathbb{T}_{k-1}$ for $k < i$. In fact, even  $M_{i\shortrightarrow k}[\mathcal{L}_i]$ is not contained in $\mathbb{T}_{k-1}$ for $k < i$.
\end{proposition}

\begin{remark}\xlabel{rTangent}
Note that, whenever $L$ is some line on $X$, $p, q$ two points on it, we have $(\sigma_p^{-1})_{\ast} (\mathbb{T}_q) = \mathbb{T}_q$: a point $t$ in $\mathbb{T}_q$ off $L$ spans a $\PP^2$ together with $L$. This intersects $X$ in $L$ and a conic through $q$, which is invariant under reflection in $p$.
\end{remark}

\begin{proof}
Each of the $\mathcal{L}_i$, in particular $\mathbb{T}_i$, contains some line through $p_i$. Fix one of the tangent divisors and one such line $L'=L'_i$.  Together with the plane through $C$ and $L$ it spans a $\PP^3$ which intersects $X$ in a cubic surface $K=K_i$. 

Note that all birational transforms
\[
M_{i\shortrightarrow k}[L']= (\sigma_k\circ \dots \circ \sigma_{i-1}) [L'] \quad k< i , k\in \ZZ , 
\]
lie on $K$. Moreover, we have
\[
\mathbb{T}_{x}X \cap K = \mathbb{T}_{x} K  \cap K \quad \forall x\in K.
\]
Let 
\[
D_{i} = \mathbb{T}_{p_i} K  \cap K \quad\mathrm{for}\; i\in\ZZ.
\]
It suffices to prove that none of the $\sigma_k\circ \dots \circ \sigma_{i-1} (L')$ is equal to a component of  $D_{k-1}$. Moreover, $M_{i\shortrightarrow k}[L']$ can never equal $L$ since $L\cup C$ is invariant under all three reflections. Let $D'_{k-1}:= \overline{D_{k-1}\setminus L}$. Notice that on the surface $K$ all three reflections are defined everywhere except in the reflection point. Hence, it suffices to check that  
\[
\sigma_k\circ \dots \circ \sigma_{i-1} (p_i) \quad k< i , k\in \ZZ,
\]
does not coincide with a point in $D'_{k-1}\cap (C\cup L)=\{ p_{k-1}, p_{k-1}' \}$. The scheme-theoretic intersection has a doublepoint in $p_{k-1}$ and a further third point $p'_{k-1}$.

We check that there is one example of a configuration $X, L, C, p, q, r$ having this property by an explicit computer calculation. To reduce the problem to a finite computation, we use the following Lemma.
\end{proof}

\begin{lemma}\xlabel{lNotExistent}
Retain the notation above. 

 If the linear map induced by 
\[
\phi:=\sigma_p\sigma_q\sigma_r\sigma_p\sigma_q\sigma_r = \sigma_0 \sigma_1 \dots \sigma_5
\]
on $L\simeq \PP^1 \simeq \PP (\CC^2)$ is diagonalizable with eigenvalues of distinct absolute value, then the point $\lambda=a$ or $\lambda=b$ is an attractor for the iterates of $\phi$.

Assume without loss of generality that $b$ is the attractor. 
Moreover, let $\mathcal{B}$ be the following set of ``bad" points on $L$: 
\begin{gather*}
(\sigma_5\sigma_4\sigma_3\sigma_2\sigma_1)(p_0), (\sigma_5\sigma_4\sigma_3\sigma_2\sigma_1)(p'_0) , \\
(\sigma_5\sigma_4\sigma_3\sigma_2)(p_1), (\sigma_5\sigma_4\sigma_3\sigma_2)(p'_1), \\
(\sigma_5\sigma_4\sigma_3)(p_2), p_2'.
\end{gather*}

Then, for each $x\in L$ with $|x - b| < \mathrm{min} \{ |y-b|  : y\in\mathcal{B} \}$, the point  
\[
\sigma_k\circ \dots \circ \sigma_{i-1} (x) \quad k< i , k\in \ZZ,
\]
does not coincide with a point in $D'_{k-1}\cap (C\cup L)=\{ p_{k-1}, p'_{k-1} \}$.

\end{lemma}

A computer calculation \cite{BBS15} shows that there exists an example of $X,L,C,p,q,r$ and three different lines $L'_i$, $i=0,1,2$, such that $\phi$ satisfies the assumptions of the Lemma, and there exists an integer $k_0$ with $k_0\equiv 0$(mod $3$) and the property that, defining
\[
x_k:=\sigma_k\circ \dots \circ \sigma_{i-1} (p_i) \quad k< i , k\in \ZZ,
\]
 $x_{k_0}\in L$ with $|x_{k_0} - b| < \mathrm{min} \{ |y-b|  : y\in\mathcal{B} \}$ and $x_k\neq p_{k-1}, p'_{k-1}$ for $k_0\le k \le i$. This concludes the proof of Proposition \ref{pGenericConfiguration}.

\begin{proof}
We divide the proof into steps. See Figure 3 for the geometric intuition. 

\textbf{Step 1.} Let $L'$ and $K$ be given; for $p\neq x\in L\subset K$, the reflection $\sigma_p$ is defined in $x$ as a map from $K$ to $K$ in the following way. Consider the conic $C_x$ defined by  $\mathbb{T}_xK\cap K =L\cup C_x$. The image point of $x$ is nothing but the second intersection $x'$ of the conic  with $L$.  It follows that for any point $x\in L$ not equal to $p$, $q$, $\sigma_p(q)$ or $\sigma_q(p)$ 
\[
(\sigma_p\circ\sigma_q )(x) =( \sigma_q\circ\sigma_p) (x) =x. 
\]

\smallskip

\textbf{Step 2.}
Note that by the preceding step, 
\[
\phi = \sigma_r\sigma_p\sigma_q\sigma_r : L \simeq \PP (\CC^2) \to L = \PP (\CC^2)
\]
is nothing but the return map to $L$. It is induced by a linear map $\CC^2\to \CC^2$. If the matrix realizing this automorphism is diagonalizable with eigenvalues of distinct absolute values, then, in an appropriate basis, it has the form
\[
\begin{pmatrix}
 \mu_1 & 0 \\ 0 & \mu_2 
\end{pmatrix}.
\]
Note that the eigenspaces are spanned exactly by $a$ and $b$ since $\sigma_q\sigma_p (a) = a$ and similarly for $b$ by Step 1. Also note that only the ratio 
\[
\frac{|\mu_1|}{|\mu_2|}
\]
is important to determine the behavior of the iterates. If $b$ is the attractor,  the matrix can be assumed to be of the form
\[
\begin{pmatrix}
 \mu & 0 \\ 0 & 1
\end{pmatrix}
\]
and $|\mu | < 1$.  

The last assertion of the Lemma follows from the fact that $\phi$ decreases distances to the attractor $b$, and by definition of the bad points, $x$ cannot get mapped to $p_{k-1}, p_{k-1}'$ by $\sigma_k\circ \dots \circ \sigma_{i-1} (x)$.
\end{proof}

\begin{figure}
\caption{Attracting behavior under iteration of reflections.}
\begin{center}
\includegraphics[width=130mm, height=100mm]{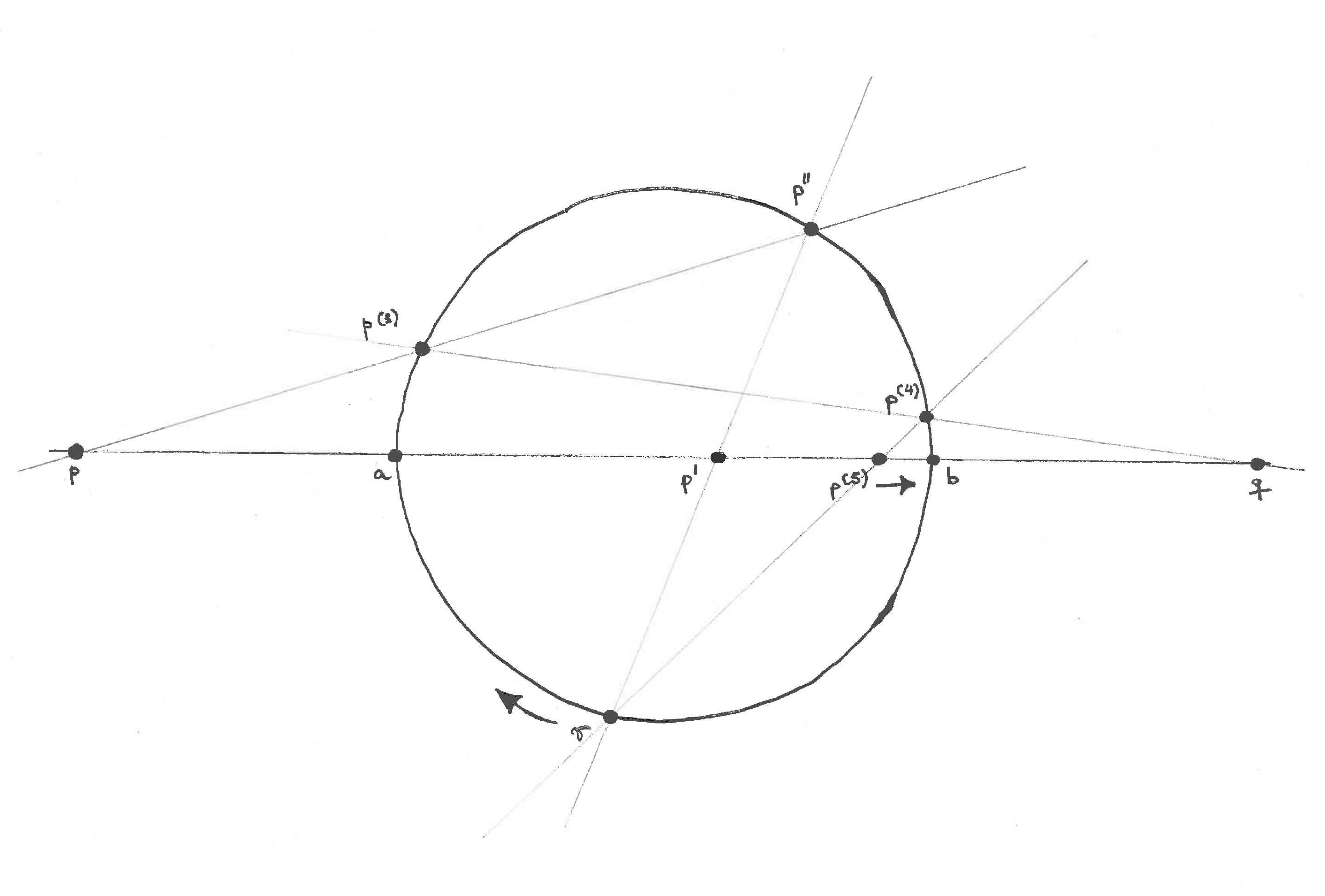}
\end{center}
\end{figure}

\subsection{Dynamics of curve germs}\xlabel{ssJets1}

Let $L$ be a line on a cubic fourfold $X$, and $x$ be a point on $L$. Furthermore, let $\Delta$ be a curve germ (in the classical topology) through $x$; let $x\neq z\in L$ be a point and consider $\sigma_z$. Then $\sigma_z[\Delta ]\cap L$ is a point determined by the normal direction to $\mathbb{T}_{\mathcal{L}_z, x}$ induced by $\Delta$ in $x$:  this follows from the fact that blowing up $z$ and the locus of lines through $z$ we obtain a morphism $\widetilde{\sigma}_z$ on this blow-up as in Section \ref{sSingleReflection}.

\begin{lemma}\xlabel{lPointsOnLine}
Let $X, L, C, p, q, r$ be a configuration on a cubic fourfold consisting of a line $L \subset X$, $C$ a conic in a $\PP^2$ with $L$ such that $L\cap C$ is transverse, $p,q$ points on $L$ away from $L\cap C$, $r$ a point on $C$ away from $L\cap C$. Then
\begin{itemize}
\item[(1)]
$\mathbb{T}_{\mathcal{L}_p, x} = \mathbb{T}_{\mathcal{L}_q, x}$ for $x \in L$, $x \neq p, q$. Here we view these spaces as embedded tangent spaces in the ambient $\PP^5$. We denote the constant two-dimensional subspace specified by the $\mathbb{T}_{\mathcal{L}_p, x}$ simply by $\Pi$ in the sequel.
\item[(2)]
Consider the ``return map to $L$" given by $F = \sigma_r\sigma_q\sigma_p\sigma_r$. For $x\in L\setminus \{ p,q\}$,
\[
dF_x (\Pi) = \Pi .
\]
\end{itemize}
\end{lemma}

\begin{proof}
We start by recalling some facts about lines on a cubic hypersurface $X^n \subset \PP^{n+1}$, see \cite[Sect.\ 6 \& 7]{CG72} or \cite[Sect.\ 1]{Izadi99}: the normal bundle of a line $l$ on $X$ can be of the following two types
\begin{gather*}
N_{l/X^n} \simeq \mathcal{O}_l\oplus\mathcal{O}_l\oplus \mathcal{O}_l(1)^{\oplus (n-3)} : \mathrm{lines}\: \mathrm{of}\: \mathrm{the}\: \mathrm{first}\: \mathrm{type;}\\
N_{l/X^n} \simeq \mathcal{O}_l(-1) \oplus \mathcal{O}_l(1)^{\oplus (n-2)} : \mathrm{lines}\: \mathrm{of}\: \mathrm{the}\: \mathrm{second}\: \mathrm{type.}
\end{gather*}
The dimension of the entire Fano variety of lines, which is smooth and irreducible, is $2(n-2)$ and the subvariety of lines of the second type is $n-2$. Moreover, for a line of the first type, the intersection of all the embedded projective tangent spaces to $X^n$ along $l$ is a linear projective subspace of $\PP^{n+1}$ of dimension $n-2$, and the same holds for a line of the second type with $n-2$ replaced by $n-1$.

In our case, this means that a generic $L$ will be of the first type, and since both $\mathbb{T}_{\mathcal{L}_p, x}$ and $\mathbb{T}_{\mathcal{L}_q, x}$ are planes contained in the intersection of all the embedded projective tangent spaces to $X$ along $L$ (since the tangent bundles of the cones $\mathcal{L}_p$ resp. $\mathcal{L}_q$ are trivialized along a ruling), we conclude that $\mathbb{T}_{\mathcal{L}_p, x} = \mathbb{T}_{\mathcal{L}_q, x} = \Pi$ is constant and equal to the intersection of tangent spaces along $L$. This proves (1).

For (2) remark that the differential $dF$ of the return map $F$ fits into a commutative diagram
\[
\xymatrix{
N_{L/X}\ar[d] \ar[r]^{dF} & N_{L/X} \ar[d] \\
L \ar[r]^F  & L
}
\]
and composing with 
\[
\xymatrix{
N_{L/X}\ar[d] \ar[r]^{\eta} & N_{L/X} \ar[d] \\
L \ar[r]^{F^{-1}}  & L
}
\]
where $\eta$ is any lift of the projectivity $F^{-1} \colon L\simeq \PP^1 \to \PP^1$ to the vector bundle $N_{L/X} = \mathcal{O}^{\oplus 2} \oplus \mathcal{O}(1)$ preserving the summands $\mathcal{O}(1)$ and $\mathcal{O}^{\oplus 2}$, we get that  $dF \circ \eta$ is a bundle automorphism of $N_{L/X}$ hence preserves the individual summands as well. Therefore, $dF$ preserves $\Pi$ which is spanned by the total space of $\mathcal{O}(1)$ and $L$.
\end{proof}

The following genericity statement is a major ingredient for justifying the computations in Theorem \ref{tConicLine} below.

\begin{proposition}\xlabel{pGenericJets}
There is a sufficiently generic choice of the configuration $X,L,C, p,q,r$ and a curve $\Gamma \in (H)^3$ such that the following holds for the moves $M_{0 \shortrightarrow \nu }[\Gamma]$, $1\le \nu < \infty$:
\begin{itemize}
\item[(1)]
Outside of $C$ and $L$, the transform $M_{0 \shortrightarrow \nu }[\Gamma]$ intersects $\mathbb{T}_{\nu}$ transversely in finitely many points which all lie outside $\mathcal{L}_{\nu}$. 
\item[(2)]
In the notation of the preceding item, let us consider a germ $\Delta$ (in the Euclidean topology) of $M_{0 \shortrightarrow \nu }[\Gamma]$ around any of the said intersection points with $\mathbb{T}_{\nu}$. Then for all $\nu +1 \le i < \infty$, the move $M_{\nu \shortrightarrow i} [\Delta ]$ is well-defined, and for $\nu+1 \le i < \infty$, it is a smooth curve germ not passing through any of the points $p,q,r$, but some point on $L$ or $C$ other than these three. 
\end{itemize} 
\end{proposition}

\begin{proof}
For (1), we use Proposition \ref{pGenericConfiguration}: we choose $\Gamma$ in such a way that it intersects all $M_{\nu \shortrightarrow 0}[\mathbb{T}_{\nu}]$ transversely in points away from $M_{\nu \shortrightarrow 0}[\mathcal{L}_{\nu}]$ and $M_{j\shortrightarrow 0}[\mathbb{T}_{j}]$ for $j<\nu$. The map
\[
(\sigma_{\nu-1}\circ \dots \circ \sigma_0)\colon X\dashrightarrow X
\]
is an isomorphism onto its image when restricted to the open $X^0\subset X$ which is the complement of the $M_{j\shortrightarrow 0}[\mathbb{T}_{j}]$ for $j<\nu$. From (2), which we prove below, it follows that all points in
\[
(\sigma_{\nu-1}\circ \dots \circ \sigma_0)[\Gamma ] \setminus (\sigma_{\nu-1}\circ \dots \circ \sigma_0)|_{X^0}(\Gamma\cap X^0)
\] 
are contained in $C\cup L$. Hence we get (1). 

\smallskip

To prove (2), for notational convenience, we will only give the proof for the case that $\Delta$ is a curve germ passing through a point of $\mathbb{T}_p$; i.e.,  $\nu =0$, but everything else is arbitrary. For the general case, we simply use Proposition \ref{pGenericConfiguration} again, but otherwise no new arguments are needed. The main point is that a $\Delta$ intersecting $\mathbb{T}_p$ sufficiently generically will verify (2).

In fact, choosing $\Delta$ generically, the birational transform $\Delta' = \sigma_p[\Delta]$ will have a generic tangent direction in $p$, i.e.\ we can realize an open dense subset of directions in $T_pX$ choosing a generic $\Delta$.  We will now consider the sequence of birational transforms
\[
\sigma_q [\Delta'], (\sigma_r\circ\sigma_q)[\Delta'], (\sigma_p\circ \sigma_r\circ\sigma_q)[\Delta'], \dots
\]
and the sequence of points $x_1, x_2, x_3, \dots $ on $L\cup C$ in which these curve germs intersect $L \cup C$. We will prove two statements about these now:

\begin{itemize}
\item[(A)]
The sequence $x_1, x_2, \dots $ depends only on the element in $\PP (N_{\mathcal{L}_p/X, p})$ which the tangent direction of $\Delta'$ in $p$ induces.
\item[(B)]
None of the points $x_1, x_2, \dots$ coincides with any of $p,q,r$ for a generic $\Delta$ resp. $\Delta'$.
\end{itemize}

To prove (A) we use Lemma \ref{lPointsOnLine}, (1). First of all, if $\Delta^{x_i}$ is a curve trait passing through $x_i$ on $L$, then clearly $x_{i+1} = \sigma_p[\Delta^{x_i}]\cap L$ and $x_{i+2} = (\sigma_q\sigma_p)[\Delta^{x_i}]\cap L$ depend only on the initial normal direction to $\mathbb{T}_{\mathcal{L}_p, x_i} =\Pi$ that $\Delta^{x_i}$ induces: by the geometry of a single reflection explained in Section \ref{sSingleReflection} (see also Figure 1 in particular), $\sigma_p$ maps $\Delta^{x_i}$ to a curve trait $\Delta^{x_{i+1}} =\sigma_p[\Delta^{x_i}]$ through a point $x_{i+1}$ on $L$ that is the image of the normal direction of $\Delta^{x_i}$ in $x_i$ under $\pi\circ \widetilde{\sigma}_p$. Moreover, the normal direction to $\Pi$ which $\Delta^{x_{i+1}}$ induces in $x_{i+1}$ is determined by the fact that it is the one that under $\pi\circ \widetilde{\sigma}_p$ gets mapped back to $x_i$: hence it is the one that $\mathbb{T}_{X, x_i}\cap \mathbb{T}_{X, x_{i+1}}$ induces in $x_{i+1}$. Then a similar argument for $\sigma_q$ shows that also $\Delta^{x_{i+2}}= (\sigma_q\sigma_p)[\Delta^{x_i}]$ along with its normal direction to $\Pi$, and in particular $x_{i+2}$, are completely determined by the normal direction of the initial curve trait $\Delta^{x_i}$. 

Next, using Lemma \ref{lPointsOnLine} (2), we see that also the normal direction of $F [\Delta^{x_{i+2}}]$ to $\Pi$ is determined by that of $\Delta^{x_{i+2}}$ hence of $\Delta^{x_i}$. This shows (A) above. 

Finally (B) follows from Proposition \ref{pGenericConfiguration} which shows that all backward transforms of tangent divisors remain divisorial. Moreover, the forward moves of tangent divisors are first a point, then curves.  Hence there is a rational map from $\mathbb{T}_p$ onto either $L$ or $C$ induced by $\sigma_i\circ\dots \circ\sigma_0$. Thus there is an open subset of $\mathbb{T}_p$ on which this map is a morphism and maps dominantly onto $L$ or $C$. If the initial curve germ $\Delta$ intersects $\mathbb{T}_p$ generically in this open subset, (B) holds. 
\end{proof}

\subsection{Determination of the dynamical degree}\xlabel{ssDegree1}

We have now assembled enough auxiliary results to justify our computations in

\begin{theorem}\xlabel{tConicLine}
For $g = \sigma_r \circ \sigma_q\circ \sigma_p$, with $p,q, r\in X$ in a plane, $p,q$ on a line on $X$ satisfying Propositions \ref{pGenericConfiguration} and \ref{pGenericJets}, we have
\[
\lambda_1 (g) = \lambda_3 (g) = \frac{5 + \sqrt{33}}{2}\approx 5.37...
\] 
\end{theorem}

\begin{proof}
We will compute $\lambda_3 (g)$ and then prove that $\lambda_1 (g) = \lambda_3 (g)$. 

Let $\Gamma =\Gamma_0$ be a curve satisfying the conclusion of Proposition \ref{pGenericJets}.

We compute the degrees of the birational transforms $\Gamma_i = M_{0\shortrightarrow i}[\Gamma]$ directly. It will turn out that the degree $\delta_i$ of the birational transform $\Gamma_i$ just depends on how many points $\lambda_{i-1}$ of $\Gamma_{i-1}$ (counted with multiplicities) lie on $L$, and how many points $\gamma_{i-1}$ of $\Gamma_{i-1}$ (counted with multiplicities) lie on $C$ in the preceding step of the iteration.

Suppose we start with some input data $(\lambda , \gamma, \delta )$. The following table summarizes how these numbers change by applying  $\sigma_p$, $\sigma_q$, $\sigma_r$ successively:

\begin{center}
\begin{tabular}{c| ccc}
 &  $\lambda$ & $\gamma$ &  $\delta$ \\ \hline 
 $\sigma_p$ & $ \delta$ & $\gamma$ & $2\delta - \lambda$ \\
 $\sigma_q$ & $2\delta - \lambda$ & $\gamma$ & $3\delta - 2\lambda$ \\
 $\sigma_r$ & $\gamma$ & $5\delta - 3\lambda$ & $6\delta -  4\lambda$
\end{tabular}
\end{center}

To justify the numbers in the first line, note that, by Proposition \ref{pGenericJets}, at this step there will be $\lambda$ points on $L$ none of which coincides with $p$ (or $q$). Moreover, by part (1) of the same proposition, there are $\delta -\lambda$ intersection points of the curve with $\mathbb{T}_p \setminus \mathcal{L}_p$ outside of $L$.   

Reflection in $p$ stabilizes $C$, and so does reflection in $q$, so $\gamma$ remains the same in the first and second steps. The degree gets multiplied by $2$, since $\sigma_p$ is given by a linear system of quadrics, and gets diminished by the number of points lying in the base locus of $\sigma_p$, i.e. $\lambda$. After application of $\sigma_p$, $\delta -\lambda$ points of $\Gamma$ get mapped to $p$, and the $\lambda$ points already on $L$ get mapped to some other points on $L$, adding to a total of $\delta$ points on $L$.

Now consider the second row. Note that by Proposition \ref{pGenericJets}, the curve intersects $\mathbb{T}_q\setminus\mathcal{L}_q$ in $(2\delta -\lambda )-\delta = \delta -\lambda$ points which get contracted into $q$ in this step. Together with the $\delta$ points already on $L$, this gives the first entry of the second row. 
 The degree changes to $2(2\delta -\lambda ) - \delta$ (twice the preceding degree diminished by the number of points lying in the base locus, i.e. in $L$).

Consider the third row. 
The map $\sigma_r$ interchanges $C$ and $L$. Hence there will then be $\gamma$ points on $L$, $(2\delta -\lambda ) + (3\delta - 2\lambda)$ points on $C$ (the number of points on $L$ in the preceding step plus the number of intersection points with $\mathbb{T}_rX$, which is the degree of the curve). Moreover, the degree of the preceding curve simply gets multiplied by $2$ by Proposition \ref{pGenericJets}, (1).

Thus the passage of the initial tuple to the next one is given by applying to the vector $(\lambda, \gamma , \delta )^t$ the matrix
\[
\begin{pmatrix}
0 & 1 & 0\\
-3 & 0 & 5\\
-4 & 0 & 6
\end{pmatrix}.
\]
By Lemma \ref{lMatrixEigenvalues}, we find $\lambda_3 (g) = (5 +\sqrt{33})/2$. 

To prove that $\lambda_1 (f) = \lambda_3 (f)$ note that, since the roles of $p$ and $q$ are interchangeable in the preceding argument, and all genericity assumptions continue to hold for the configuration $X, L,C, p, q$ after interchanging the roles of $p$ and $q$, 
\[
\lambda_3 (\sigma_r \circ \sigma_q \circ \sigma_p) = \lambda_3 (\sigma_r \circ \sigma_p \circ \sigma_q) 
\]
and since dynamical degrees are invariants for birational conjugacy
\[
\lambda_3 (\sigma_r \circ \sigma_p \circ \sigma_q) = \lambda_3 (\sigma^{-1}_r\circ  (\sigma_r \circ \sigma_p \circ \sigma_q)\circ  \sigma_r). 
\]
But $\sigma_p \circ \sigma_q\circ  \sigma_r$ is the inverse of $\sigma_r \circ \sigma_q \circ \sigma_p$, hence by Lemma \ref{lInverse}, 
\[
\lambda_3 (\sigma_r \circ \sigma_q \circ \sigma_p)= \lambda_3 (\sigma_p \circ \sigma_q\circ  \sigma_r ) = \lambda_1 ( \sigma_r \circ \sigma_q \circ \sigma_p ).\qedhere
\]
\end{proof}

\begin{remark}\xlabel{rSecond}
We suspect that in this case also $\lambda_2(g)= \frac{5 + \sqrt{33}}{2}\approx 5.37...$. Conditional on some genericity assumptions, which, unfortunately, we have not yet been able to show are always realizable at the same time, we can prove this; the result is also supported by independent extensive computer calculations.
\end{remark}

\begin{remark}\xlabel{rRatioFirstSecondDegree}
One should not be left with the impression that it is reasonable to suspect the equality $\lambda_1 (g) = \lambda_2 (g) = \lambda_3 (g)$ for every $g$ in the subgroup of $\mathrm{Bir}(X)$ generated by reflections, let alone for $g$ in all of $\mathrm{Bir}(X)$. For instance, in the case of two points on a line and $N$ points general outside of that line, we think that $\lambda_2 \neq \lambda_1$, but cannot yet prove it. Certainly, obtaining bounds on the overall variance from its mean of the tuple $(\lambda_1(g), \lambda_2(g), \lambda_3 (g))$, for $g$ ranging over $\mathrm{Bir}(X)$, seems to be a main question for proving irrationality of a very general $X$ by this type of quantitative refinement of the Noether-Iskovskikh-Manin approach.
\end{remark}

\section{A triangle of lines}\xlabel{sTriangle}

Here we discuss another interesting geometric configuration of three special points on $X$.

Let $p,q,r$ be three distinct points on $X$ which form the vertices of a triangle of lines on $X$. We again write $p_0:=p, p_1:=q, p_2:=r$, and $L_{p_ip_j}$ for the line joining $p_i$ and $p_j$. We also retain the notation $\mathbb{T}_i$ for the tangent hyperplane section in $p_i$ and write once more 
\[
g = \sigma_r \circ \sigma_q \circ \sigma_p .
\]
We will compute the first and third dynamical degrees of $g$. The strategy follows roughly the steps set down in Subsection \ref{ssComputation}.

We start with a Lemma about matrices, which will be used in the proof of Theorem \ref{tTriangle}. 

\begin{lemma}\xlabel{lMatrices}
Let 
\begin{gather*}
P_0 = 
\begin{pmatrix}
2 & 0 & 0 & -1 & -1 & 0 \\
1 & 1 & 0 & -1 & -1 & 0 \\
1 & 0 & 1 & -1 & -1 & 0 \\
1 & 0 & 0 & 0 & -1 & 0 \\
1 & 0 & 0 & -1 & 0 & 0 \\
0 & 0 & 0 & 0 & 0 & 0
\end{pmatrix} ,
\quad 
P_1 = \begin{pmatrix}
1 & 1 & 0 & -1 & 0 & -1 \\
0 & 2 & 0 & -1 & 0 & -1 \\
0 & 1 & 1 & -1 & 0 & -1 \\
0 & 1 & 0 & 0 & 0 & -1 \\
0 & 0 & 0 & 0 & 0 & 0 \\
0 & 1 & 0 & -1 & 0 & 0
\end{pmatrix}, \\
P_2 = \begin{pmatrix}
1 & 0 & 1 & 0 & -1 & -1 \\
0 & 1 & 1 & 0 & -1 & -1 \\
0 & 0 & 2 & 0 & -1 & -1 \\
0 & 0 & 0 & 0 & 0 & 0 \\
0 & 0 & 1 & 0 & 0 & -1\\
0 & 0 & 1 & 0 & -1 & 0 
\end{pmatrix}.
\end{gather*}
Also, as usual, for $n \in \ZZ$, put $P_n := P_j$ for that $j \in \{ 0,1,2\}$ with $n \equiv j$ (mod $3$). For $i \ge 0$, consider the product
\[
A_i := P_i P_{i-1} \dots P_1 P_0. 
\]
For a vector $v\in \ZZ^6$ denote its $k$-th coordinate by $v_k\in \ZZ$. We start numbering components with zero. Consider $v^{(0)} = (1,0,0,0,0,0)^t$ and $v^{(i+1)} = A_i v^{(0)}$. Moreover, consider the ideals in $\CC [x_0, \dots , x_5]$ given by
\begin{gather*}
I_0 := \langle (x_0, x_1, x_2) \cdot (x_0, x_1, x_2, x_3, x_4), x_3x_4, x_0x_5 \rangle , \\
I_1 := \langle (x_0, x_1, x_2) \cdot (x_0, x_1, x_2, x_3, x_5), x_3x_5, x_1x_4 \rangle , \\
I_2 := \langle (x_0, x_1, x_2) \cdot (x_0, x_1, x_2, x_4, x_5), x_4x_5, x_2x_3 \rangle . 
\end{gather*}
The ideal $I_l$ is generated by a space $M_l$ of quadratic monomials. Similarly to the $P_n$, we also define $I_n$ for $n\in \ZZ$. Now for each monomial $x_{\mu}x_{\nu} \in M_{i+1}$, we consider the sum of vector components $v^{(i+1)}_{\mu}+v^{(i+1)}_{\nu} \in \ZZ$, and the pair $(\mu_0, \nu_0)$, for which this integer is minimal. Then
\begin{gather*}
(\mu_0, \nu_0) = (3,4) \; \mathrm{for} \; i \equiv 2 (3), \quad (\mu_0, \nu_0) = (3,5) \; \mathrm{for} \; i \equiv 0 (3), \\
(\mu_0, \nu_0) = (4,5) \; \mathrm{for} \; i \equiv 1 (3).
\end{gather*}  
\end{lemma}

\begin{proof}
From the structure of $I_l$, and more specifically $M_l$, and the position of the zero rows in $P_0, P_1$ resp. $P_2$, one sees that it is sufficient to prove for all $i$ that each of $v^{(i+1)}_0, v^{(i+1)}_1, v^{(i+1)}_2$ is greater than or equal to each of $v^{(i+1)}_3, v^{(i+1)}_4, v^{(i+1)}_5$. This is proved by induction. For example, suppose $d^{(0)}=(d^{(0)}_0, \dots , d^{(0)}_5)^t \in \ZZ^6$ is a vector for which this holds, and let us show that it also holds for $d^{(1)}=P_0 \cdot d^{(0)}$. This is an immediate consequence of the inequalities of the hypothesis, e.g. $d^{(1)}_0 = 2 d_0^{(0)} - d_3^{(0)} - d_4^{(0)} \ge d_0^{(0)} - d^{(0)}_4 = d_3^{(1)}$ since $d_0^{(0)} \ge d_3^{(0)}$. We omit the (mechanical) verification of all possible cases.
\end{proof}

Now choose coordinates $x_0, \dots , x_5$ in $\PP^5$ such that
\begin{gather*}
p_0 = (0:0:0:0:0:1), \; p_2= (0:0:0:0:1:0), \; p_2=(0:0:0:1:0:0) \\
\mathrm{and}\; \mathbb{T}_{p_i} X = \{ x_i =0 \}.
\end{gather*}

In these coordinates, we can write

\begin{align*}
\sigma_{p_0} = & (x_0^2: x_0x_1: x_0x_2: x_0x_3: x_0x_4: Q_0) \; \mathrm{with} \; Q_0 \in  M_0, \\
\sigma_{p_1} = & (x_1x_0: x_1^2: x_1x_2: x_1x_3: Q_1: x_1x_5) \; \mathrm{with} \; Q_1 \in  M_1, \\
\sigma_{p_2} = & (x_2x_0: x_2x_1: x_2^2: Q_2: x_2x_4: x_2x_5) \; \mathrm{with} \; Q_2 \in  M_2 .
\end{align*}

\begin{lemma}\xlabel{lDynamicsCurveTraits}
Consider a curve trait $\gamma_0$ transverse to $\mathbb{T}_{p_0}$ given by $(f^{(0)}_0:f_1: \dots : f^{(0)}_5)$ where $f^{(0)}_j$ is a power series in a local parameter $t$ of degree $d_j^{(0)}=1$ for $j=0$ and $d_j^{(0)} =0$ otherwise. Then, for $d^{(0)} : = (d_0^{(0)} , \dots , d^{(0)}_5)^t$ and $i \ge 0$, the trait
\[
\gamma_{i+1}:=\sigma_{p_i} \circ \dots \circ \sigma_{p_0} [\gamma_0]
\]
is well-defined and given by local power series $(f^{(i+1)}_0: \dots : f^{(i+1)}_5)$ with a degree vector $d^{(i+1)} = A_i d^{(0)}$.
\end{lemma}

\begin{proof}
Formally substituting $(f^{(0)}_0: \dots : f^{(0)}_5)$ for $(x_0: \dots : x_5)$ in the formula for $\sigma_{p_0}$, we obtain a  power series of degrees
\[
(2d_0^{(0)} , \, d_0^{(0)} +d_1^{(0)} , \, d_0^{(0)} + d_2^{(0)} , \, d_0^{(0)}+d_3^{(0)}, \, d_0^{(0)}+d_4^{(0)}, \, d_3^{(0)} + d_4^{(0)})
\]
where we used Lemma \ref{lMatrices} to justify the last entry $d_3^{(0)} + d_4^{(0)}$ in this vector. Moreover, again by Lemma \ref{lMatrices}, all entries in this vector preceding the last one are bigger than or equal to the last one. This means that we can divide $(f^{(0)}_0: \dots : f^{(0)}_5)$ by $t^{d^{(0)}_3+d^{(0)}_4}$ to obtain local power series for the strict transform 
\[
\sigma_0 [ \gamma_0] .
\]
Accordingly, these have degrees 
\begin{gather*}
(2d_0^{(0)} , \, d_0^{(0)} +d_1^{(0)} , \, d_0^{(0)} + d_2^{(0)} , \, d_0^{(0)}+d_3^{(0)}, \, d_0^{(0)}+d_4^{(0)}, \, d_3^{(0)} + d_4^{(0)})\\
 - (d_3^{(0)}+d_4^{(0)}) (1,1,1,1,1,1). 
\end{gather*}
Hence the formula for $P_0$; using Lemma \ref{lMatrices} repeatedly, we obtain the full assertion of Lemma \ref{lDynamicsCurveTraits}. 
\end{proof}

\begin{theorem}\xlabel{tTriangle}
We have, for $p,q,r$ the vertices of a triangle of lines on $X$,
\[
\lambda_1 (g) = \lambda_3 (g) = \left( \frac{1+ \sqrt{5}}{2}\right)^3 = 4.236...
\]
\end{theorem}

\begin{remark}\xlabel{rSuperstition}
If you are into number mysticism, it will not have escaped you that $\frac{1+ \sqrt{5}}{2}$ is the Golden Ratio.
\end{remark}

\begin{proof}[Proof of Theorem \ref{tTriangle}]
The fact that $\lambda_1 = \lambda_3$ follows from symmetry. 

Multiplying the three matrices $P_0, P_1, P_2$ we obtain 
\[
P:=\begin{pmatrix}
3 & 1 & 1 & -3 & -2 & 0 \\
2 & 2 & 1 & -3 & -2 & 0 \\
2 & 1 & 2 & -3 & -2 & 0 \\
0 & 0 & 0 & 0 & 0 & 0 \\
1 & 0 & 1 & -1 & -1 & 0 \\
1 & 1 & 1 & -2 & -1 & 0
\end{pmatrix}.
\]
The matrix $P$ has minimal polynomial 
\[
x^2 (x-1)^2 (x^2 - 4x -1)
\]
and the root of the last factor with the largest absolute value is 
\[
2 + \sqrt{5} = \left( \frac{1+ \sqrt{5}}{2}\right)^3 .
\]
Thus, to finish the proof, it suffices to show that the growth behavior of the degrees of the power series defining the branch $\gamma_{i+1}$ coincides with the growth behavior of the degrees of birational transforms $\Gamma_{i+1}$ of a very general curve $\Gamma_0$ in $H^3$ on $X$ under the evolution of the dynamical system. This will follow from the following

\textbf{Claim}: after application of 
\[
\sigma_{p_0}\sigma_{p_2}\sigma_{p_1}\sigma_{p_0}
\]
all subsequent birational transforms $\Gamma_i$ have no intersection points with any of $\mathbb{T}_{p_i}$ that lie outside of the plane $\Lambda=\langle p_0, p_1, p_2\rangle$, and in each step, all the intersection points are concentrated in $p_0, p_1$ or $p_2$.

If the claim is true, the proof is complete, since then the growth behavior of the degrees of the birational transforms is the same as the one of the degrees of the power series defining the branches, since both grow as the intersection multiplicities of $\Gamma_i$ with $\{ x_0=0\}$, $\{ x_1=0 \}$ resp. $\{ x_2 = 0 \}$. 

The claim, however, follows directly from two facts: (1) the tangent divisors $\mathbb{T}_{p_i}$ are invariant under all three reflections since any two points lie on a line; (2) by the formulas in Lemmas \ref{lMatrices} and \ref{lDynamicsCurveTraits}, every trait $\gamma_{i+1}$ has center in the plane after application of $\sigma_{p_0}\sigma_{p_2}\sigma_{p_1}\sigma_{p_0}$ (a priori, it might get ``pushed outside" again in case it passes through $p_i$ when $\sigma_{p_i}$ is the next transformation to be applied). These two facts imply that after one application of $\sigma_{p_2}\sigma_{p_1}\sigma_{p_0}$ all intersection points of $\Gamma_0$ with $\mathbb{T}_{p_0}$, $\mathbb{T}_{p_1}$ and $\mathbb{T}_{p_2}$ are contracted inside the plane $\Lambda$, i.e., in the sequel there are no intersection points of any of the birational transforms with a $\mathbb{T}_{p_i}$ outside of $\Lambda$. The assertion about the concentration of the intersection points in $p_0, p_1$ or $p_2$ also follows from the formulas in Lemmas \ref{lMatrices} and \ref{lDynamicsCurveTraits}.
\end{proof}

\end{document}